\documentclass{amsart}
\usepackage{euscript}           
\usepackage{amsmath,amsthm}     
\usepackage{amssymb}            

\newfont{\oldgerman}{eufm10}
\usepackage {epsf}

\newcommand {\Q}{{\mathbb Q}}

\newcommand {\Z}{{\mathbb Z}}

\newcommand{\ra}{\rightarrow}                   
\newcommand{\lra}{\longrightarrow}              

\newfont{\german}       {eufm10 at 12pt}

\newtheorem{thm}{Theorem}[section]
\newcounter{numerierer}
\newcounter{leer}

\newtheorem{defn}[thm]{Definition}
\newtheorem{prop}[thm]{Proposition}

\newtheorem{cor}[thm]{Corollary}
\newtheorem{lemma}[thm]{Lemma}

\theoremstyle{definition}  
\newenvironment{definition}{\begin{defn}\rm}{\end{defn}}

\newtheorem{remark}[thm]{Remark}
\usepackage[frame,matrix,curve, arrow]{xy}

\xymatrixcolsep{1.9pc}                          
\xymatrixrowsep{1.9pc}
\newdir{ >}{{}*!/-5pt/\dir{>}}                  
\newdir{ |}{{}*!/-5pt/\dir{|}}                  

\begin{document}

\title{Toda brackets and  congruences of modular forms }
\author{Gerd Laures}

\address{ Fakult\"at f\"ur Mathematik,  Ruhr-Universit\"at Bochum, NA1/66,
  D-44780 Bochum, Germany}
\date{\today}
\begin{abstract}
This paper investigates the relation between Toda brackets and congruences of modular forms. It determines the $f$-invariant of Toda brackets and thereby generalizes the formulas of J.F.\ Adams for the classical $e$-invariant to the chromatic second filtration. \end{abstract}
\maketitle
\section{Introduction}
In this work congruences between modular forms are used to compute Toda brackets in the second Adams-Novikov filtration. The translation from stable homotopy classes between spheres and modular forms is based on the tertiary invariant 
$$ f: \xymatrix{\pi^{st}_{2n-2}\ar[r]& (D/M_0+M_n)_{\Q/ \Z}}$$
defined in \cite{MR1660325} with values in Katz's ring of divided congruences. The $f$-invariant generalizes the classical $e$ invariant of Frank Adams. The purpose of this work is an extension of Adams' formulas \cite[§11]{MR0198470} for the $e$-invariant of Toda brackets to the second filtration:
\begin{eqnarray*}
e\left< \alpha, p\iota , \beta \right> &=& -pe(\alpha)e(\beta)\\
e\left<  p\iota, \alpha , \beta \right> &=& -p\delta e(\alpha)e(\beta)\\
e\left< \alpha , \beta , p\iota \right> &=& -p\delta e(\alpha)e(\beta)
\end{eqnarray*}
modulo the indeterminacy.
Here, $\delta$ is a certain number which only depends on the dimensions.\par
The $e$-invariant is not strong enough to detect elements in the second filtration. The reason is the following: the $e$-invariant is based on $K$-theory and in the $K$-based Adams-Novikov spectral sequence  there aren't any classes in the second line other than classes in the image of $J$. In \cite{MR1660325} $K$-theory was replaced by elliptic cohomology theory of level $N$   (see \cite{MR1071369}\cite{MR1235295}\cite{MR1271552}). It was shown that the corresponding tertiary invariant $f$ is injective on the 2-line.
So, it is natural to ask for a precise relationship between the $f$-invariant and Toda brackets. \par
The 2-line of the classical Adams-Novikov spectral sequence was computed for odd primes by Miller, Ravenel and Wilson in \cite{MR0458423} with the help of the chromatic spectral sequence and by Shimomura for $p=2$ in \cite{MR635034}. In contrast to the 1-line it is not cyclic. The number of generators is unbounded when considered as a function of the dimension. Its complicated structure makes it hard to determine classes such as Massey products. The $f$-invariant simplifies the 2-line in two steps: first the complex orientation gives an injective map into the elliptic based $E_2$ term and in the second step the extension group is mapped to Katz's ring of divided congruences. In this paper it is reinterpreted as a coboundary of a resolution which can be viewed as an integral version of a chromatic resolution.\par 
The paper starts with a recollection on higher homotopy invariants in general. Here, the edge homomorphisms are used to construct invariants for stable homotopy classes with the help of any theory $E$. Under mild conditions the first few invariants take values in the $E$-based Adams-Novikov $E_2$-term. Then the relation between Toda brackets and Massey products is explained. Adams considered Massey products for extensions of exact sequences. However, it turns out that the cobar complex is better suited to keep the homological algebra small for higher filtrations. This holds especially when  a $\Q /\Z$- reduction is used to come down by one filtration. It follows a computational  section where the commutation between the coboundary operator and  Toda brackets  is investigated.  Proposition \ref{image} gives a precise translation of the elliptic  Adams-Novikov $E_2$ term by primitives in the ring of divided congruences. It comes from a resolution which is closely related to the chromatic resolution. Finally, the desired formulas for the $f$-invariants of Toda brackets are given. In the last section some examples are calculated including the Kervaire class in dimension 30 as 4-fold Toda bracket.

\section{Higher Homotopy invariants}
Suppose $E$ is a flat ring spectrum. Let $\bar{E}$ be the fibre of the unit map $S^0\lra E$. Then there is spectral sequence $(E_r,d_r)$ with the associated filtration of $\pi_*^S$ given for $s\geq 0$ by
$$ F^s = \mbox{im} (\pi_* \bar{E}^s \lra \pi_* S^0)$$
The differential $d_r$ raises the filtration by $r$. Hence  there are well defined `edge homomorphisms'
$$e_s: \xymatrix {F^s \ar[r]&\ F^s/F^{s+1} \ar@{>->}[r]&E_\infty^s \ar@{>->}[r]& E^s_{s+1}}.$$
For a stable class $\alpha$ the value $e_s(\alpha)$ is defined if and only if $e_r(\alpha)$ vanishes for all $r< s$. It associates to $\alpha$ its representative in the $E^s_{s+1}$-term.  The invariants are natural with respect to maps between spectral sequences. 
\par
In case $E$ is complex bordism $MU$ this is the classical Adams-Novikov spectral sequence. Its quotient group $F^0/F^1$ is concentrated in dimension 0 and the invariant 
$$ e_0:\xymatrix{\pi_0^S=F^0\ar[r] &\pi_0(MU)=\Z}$$ 
takes values in the integers. It associates to a stable self map of $S^0$ its degree $d$. In fact we may replace $E$ by any other spectrum with the property that the Hurewicz homomorphism is injective in dimension 0. If $E$ has torsion coefficients the map $e_0$ can be non trivial in positive dimensions. For instance this happens for real $K$-theory in dimensions $8k+1$ and $8k+2$.
\par
The invariant
$e_1$ 
 has been studied by Adams for real and complex $K$-theory in \cite{MR0198470}. It can be described in terms of exact sequences as follows: 
 Suppose the $e_0$ invariant of $\alpha\in \pi_n^S $ vanishes and let $C_\alpha$ be the cofibre of $\alpha$. Then the sequence 
 $$ \xymatrix{E_*S^0 \ar@{>->}[r]& E_*C_\alpha \ar@{->>}[r]& E_*S^{n+1}}$$
 of $E_*E$-comodules is short exact and hence determines an element in $$E^{1,n+1}_2\cong\mbox{Ext}^{1,n+1}_{E_*E}(E_*,E_*).$$
 Adams determines the extension group for $K$-theory with the help of a monomorphism
 $$\iota:  \xymatrix{E_2^{1,n+1} \ar@{>->}[r]&\Q/\Z}.$$
 In particular, the extension term is cyclic. For $K$-theory its order in dimension $n=4t-1$ is the denominator of the divided Bernoulli number $B_t/2t$. The composite $\iota e_1$ is called the classical  $e$-invariant.  \par
For variations of the topological modular forms spectrum $TMF$ the invariant $e_2$ has been studied (in chronological order) by the author \cite{MR1660325}\cite{MR1781277}, by Hornborstel-Naumann \cite{MR2354323}, by Behrens-Laures in \cite{MR2544384} and by von Bodecker \cite{vB09a}\cite{vB09c}. For its relation to characteristic classes of manifolds with corners see Laures \cite{MR1781277} and for an interpretation as a spectral invariant see von Bodecker \cite{vB09b} and Bunke-Naumann \cite{BN09}. 
\par
Before describing the invariant we first take a look at the case of complex bordism.  Here, for positive even $n$ the $e_1$-invariant vanishes. Moreover, there are no $d_2$-differentials hitting the 2-line. Thus the $E_3^2$-term injects into the $E_2^2$ term. This yields a well defined homomorphism
$$ e_2: \xymatrix{\pi_n^S\ar[r] & E_2^{2,n+2}}=\mbox{Ext}^{2,n+2}_{E_*E}(E_*,E_*).$$ 
Let us describe this invariant in terms of exact sequences in more detail: observe that any lift $\bar{\alpha}$ of a stable class $\alpha$ is unique since $E_*$ and $E_*E$ are evenly graded. Moreover, the sequence
$$  \xymatrix{E_*\Sigma \bar{E} \ar@{>->}[r]& E_*\Sigma C_{\bar{\alpha}} \ar@{->>}[r]& E_*S^{n+2}}$$
is short exact and can be spliced together with 
$$  \xymatrix{E_*S^{0} \ar@{>->}[r]& E_* E \ar@{->>}[r]& E_*\Sigma\bar{E}}$$
to get an extension of degree 2.  \par
Next suppose $E$ is complex oriented with $(p,v_1)$  regular and $v_2$ invertible mod $(p,v_1)$. Then locally at $p$ the 2-line of $E$ coincides with the 2-line of $v_2^{-1}MU_{(p)}$ \cite[4.3.3]{MR1781277} and there is a well understood injection\cite{MR860042}
$$ \xymatrix{E^2_2(MU_{(p)}) \ar@{>->}[r]&E^2_2(v_2^{-1}MU_{(p)})\cong E_2^2 (E)}.$$
In case of $\Gamma_1(N)$ topological modular forms $E=TMF_1(N)$ there is an injection 
 $$ \xymatrix{E_2^{2,2t} \ar@{>->}[r] & (D/M_0+M_t)_{ \Q /\Z }}$$
where $D$ is the ring of divided congruences and $M$ the ring of modular forms. This map will be reviewed in section \ref{invariants}. Furthermore, proposition \ref{image} describes its precise image.  
\par The $e_3$-invariant is defined for odd dimensional homotopy classes in the cokernel of the $J$-homomorphism. It takes values in the $E_2$-term since all elements in the 1-line are permanent cycles.  Hence the invariant
$$e_3: \xymatrix{\mbox{coker} J\ar[r]& E_2^3=\mbox{Ext}^3_{E_*E}(E_*,E_*)} $$
is well defined. It will not be further investigated here.
\section{Toda brackets and Massey products}
In the sequel we assume that $E$ is a complex oriented flat ring theory. The $E_1$-term of the $E$-based Adams-Novikov spectral sequence is the cobar complex $\Omega_*$ which is a differential graded algebra. We use the notation $\Gamma$ for the comodule $E_*E$ over $A=E_*$ and $\bar{\Gamma}$ for the augmentation ideal $E_*\Sigma\bar{E}$. Then the differential 
$$d_1: \xymatrix{\bar{\Gamma}^{\otimes s} \ar[r]&\bar{\Gamma}^{\otimes (s+1)}}$$
sends $[\alpha_1|\cdots |\alpha_s]$ to $[1|\alpha_1|\cdots |\alpha_s] -[\alpha_1'|\alpha_1''|\cdots |\alpha_s] +\ldots +(-1)^{s+1}[\alpha_1|\cdots |\alpha_s|1]$ and the product is given by concatenation. 
In fact, each $(E_r,d_r)$ is a differential graded algebra in which Massey products can be defined. 
Massey products in the cobar complex are related to Toda brackets in stable homotopy.  This relation has been studied by Adams, Moss, Lawrence, May, Kochman and others and is well known. \par
Recall from \cite{MR1407034},
\begin{definition}
Suppose $u\in E_r^{s,t+s}$ and $w\in E_{r'}^{s',t+s'}$ with $s'<s$. Then $d_{r'}(w)$ is a crossing differential of $d_r(u)$ if $s'+r'>s+r$.
\end{definition}
The following result reformulates the result \cite[5.7.5]{MR1407034} for the homotopy invariants of the last section. 
\begin{prop}\label{toda-massey}
Suppose that the Toda bracket $\left<\alpha_1, \ldots ,\alpha_n \right>$ is strictly defined. Let $ a_i$ be a representative of   $e_{s(i)}(\alpha_i)$  in the $MU$-based $E^{s(i)}_{r+1}$ for some $r$. Suppose further that the Massey product $\left< a_1,\ldots, a_n\right>$ is defined in $E_{r+1}$ and that there are no crossing differentials of $d_r a_{ij}$ for all defining systems $(a_{ij})$.

Then for $s=\sum_is(i)$ there is a class $\alpha \in \left<\alpha_1,\ldots ,\alpha_n \right>$ of filtration $s-n+2$ with 
$$ e_{s-n+2}(\alpha)\in \left< a_1,\ldots, a_n\right>.$$
\end{prop} 
\begin{proof}
Since $a_i$ converges to $\alpha_i$ we can apply of \cite[5.7.5]{MR1407034}. Hence $\left< a_1,\ldots, a_n\right>$ consists of infinite cycles which converge to elements of $\left<\alpha_1, \ldots ,\alpha_n \right>$. Since  all elements have filtration $s-n+2$ the claim follows. 
\end{proof}
\begin{cor}
Under the assumptions of \ref{toda-massey} we have
$$ e_{s-n+2}(\alpha)\in \left< a_1,\ldots, a_n\right>$$
for all spectral sequences based on a complex oriented theory $E$.
\end{cor}
\begin{proof}
This follows from the naturality of all constructions.
\end{proof}

The proposition suggests that for the $e_k$-invariant of  $n$-fold Toda brackets one should consider only sums of filtrations $s=k+n-2$. The cases for $n=3$ and $k=0,1$ were considered by Adams in \cite[Theorem 5.3]{MR0198470}. 
\par
The case $k=2$ is the object of the rest of the paper. In the next section triple Massey products are considered. Note that in this case there aren't any crossing differentials in the defining system as long as there aren't any $r$-boundaries in the 3-line for $r\geq 2$. 
Moreover, if the Toda bracket $\left< \alpha_1,\alpha_2, \alpha_3\right>$ is defined then so is the Massey product $\left< a_1,a_2, a_3\right>$. 

\begin{cor}\label{indet}
\begin{enumerate}
\item
Suppose that the Toda bracket   $\left< \alpha_1,\alpha_2, \alpha_3\right>$ is defined. Then 
$$ e_2 \left< \alpha_1,\alpha_2, \alpha_3\right>=\left< e_{s_1}\alpha_1,e_{s_2}\alpha_2, e_{s_3}\alpha_3\right> $$
modulo
$$ e_{s_1}\alpha_1  H^{s_2+s_3-1,|\alpha_2|+s_2+|\alpha_3|+s_3}(\Omega_*)+
H^{s_1+s_1-1,|\alpha_1|+s_1+|\alpha_2|+s_2}(\Omega_*)e_{s_3}\alpha_3.$$
\item
Suppose that the Toda bracket $\left< \alpha_1,\alpha_2, \alpha_3,\alpha_4\right>$ is strictly defined and that there are no crossing differentials. Then 
$$ e_2 \left< \alpha_1,\alpha_2, \alpha_3,\alpha_4\right>=\left< e_{s_1}\alpha_1,e_{s_2}\alpha_2, e_{s_3}\alpha_3,e_{s_4}\alpha_4 \right>$$
modulo
$$
\sum_{a,b,c}\left< a, e_{s_3}\alpha_3,e_{s_4}\alpha_4\right>+\left<  e_{s_1}\alpha_1,b,e_{s_4}\alpha_4\right>+\left<  e_{s_1}\alpha_1,e_{s_2}\alpha_2,c\right>.$$
\end{enumerate}
\end{cor}
\begin{proof}
Since the invariant $e_2 $ is multiplicative it maps the indeterminacy of the Toda bracket
$$\alpha_1 \pi^S_{|\alpha_2|+|\alpha_3|+1}+\pi^S_{|\alpha_1|+|\alpha_2|+1}\alpha_3$$
to the indeterminacy of the Massey product stated above. The same holds for the 4-fold products for which the indeterminacy  can be found for the Toda bracket in \cite[theorem 2.3.1]{MR1052407} and  for the Massey product in \cite[proposition 2.4]{MR0238929}.
\end{proof}
\section{The $\Q /\Z$-reduction}
Let  $(A,\Gamma)$ be a Hopf algebroid. In this section we assume that $(A,\Gamma)$ has the following properties:
\begin{enumerate}
\item $\Gamma$ is flat as an $A$-module.
\item $A$ and $\Gamma$ are torsion free.
\item The map $\phi: A^{\otimes 2}_\Q \rightarrow \Gamma_\Q$ which sends $a\otimes b$
to $a\eta_R(b)$ is an isomorphism.
\end{enumerate}
Define the cosimplicial abelian group $\Omega_\Q^n=A_\Q^{n+1}$ with cofaces
$$\partial^i (a_0\otimes \ldots \otimes a_n)=a_0\otimes  \ldots \otimes a_{i-1}\otimes 1\otimes a_{i+1}\otimes \ldots\otimes a_n.$$
The notation is justified to the fact that the map $\phi^{\otimes n}$ provides an isomorphism between the complex $\Omega_\Q$ and the rationalized cobar complex. Its cohomology is concentrated in dimension 0 
$$H^*(\Omega_\Q^*)=H^0(\Omega_\Q^*)=\Q.$$
Its algebra structure is given by
$$(a_0\otimes \ldots \otimes a_n)\otimes(b_0\otimes \ldots \otimes b_m)\mapsto a_0\otimes \ldots \otimes a_{n-1} \otimes a_nb_0\otimes b_1\otimes \ldots \otimes b_m.$$
The short exact sequence
$$ \xymatrix{\Omega^*\ar@{>->}[r]& \Omega^*_\Q \ar@{->>}[r]& \Omega^*_{\Q/\Z}}$$
yields a connecting isomorphism
$\delta : H^{n-1}(\Omega^*_{\Q/\Z}) \ra H^{n}(\Omega^*_{})$ in positive dimensions. Choose an augmentation 
$\tau: \Omega_\Q \lra \Q$. Then the following observation is easily verified.
\begin{lemma}
Let $a\in \Omega^n$ be a cycle and suppose $\sum a_0\otimes \ldots \otimes a_n\in \Omega^{n}_\Q$ is its rationalization. Then 
we have
$$ \delta^{-1} [a]=[\sum \tau(a_0)a_1\otimes a_2 \otimes \ldots \otimes a_n].$$ 
Moreover, when writing $b'=\sum \tau(b_0) b_1 \otimes \ldots  \otimes b_n$  for  $b \in \Omega^n_\Q$ we have   $$  (db)' =b-db'.$$
\end{lemma}
We are going to study the image of triple Massey products $\left< x,y,z\right> $ under $\delta^{-1}$. For $n=2$ there are 3 cases of interest: $(|y|=0,|x|+|z|=3)$,  $(|y|=1,|x|+|z|=2)$ and $(|y|=2, |x|+|z|=1)$. We start with the first case.
\begin{definition}
Suppose $\delta [a]\in H^2(\Omega^*)$ is represented by some  $a\in A_\Q \otimes A_\Q$.
We say that  $a$ is $q$-adapted for some $q\in \Z$ if $qa \in \Gamma$.
\end{definition}
\begin{lemma}\label{4.3}
$q$-torsion classes admit $q$-adapted representatives.
\end{lemma}
\begin{proof}
Choose a representative $a $ which is not yet $q$-adapted. Since $\delta [a]$ is q-torsion  the element  $qa$  represents a boundary in $\Omega^1_{\Q /\Z}=\Gamma_{\Q/\Z}$. Hence there is $r\in A_\Q$ with 
$$ \tilde{a}=q a +dr \in \Gamma.$$
and $\tilde{a}/q$ is a $q$-adapted representative.
\end{proof}
\begin{prop}\label{4.4}
Suppose that for some $q\in \Z$ the Massey product $\left<\delta [a],q,\delta[c]\right>$ is defined.
\begin{enumerate}
\item
For $|[a]|=1, |[c]|=0$ and $a$ $q$-adapted we have
$$\delta  [q a c]\in \left<\delta  [a],q,\delta [c]\right>.$$
\item
For $|[a]|=0, |[c]|=1$ and $c$  $q$-adapted we have
$$ \delta [-q a c]\in \left< \delta [a],q,\delta [c]\right>.$$
\end{enumerate}
\end{prop}
\begin{proof} 
For $(i)$ observe that
$\delta [a]$ is represented by $d a \in \Gamma^{\otimes 2}$. Hence $q\, da$ is the boundary of $q a\in \Gamma$. Similarly, let $dc\in \Gamma$ then $q\, dc$ is the boundary of $ qc\in A$. Thus we obtain the representative of the Massey product
$$ [da\, qc-qa \, dc]\in \left< \delta [a],q,\delta [c]\right>.$$ The lemma tells us that the desired class is obtained by the applying $\tau$ to the first factor in its tensor product expression. Alternatively, one writes the representative as $d(qac)$ and obtains the result.
The second formula follows from a similar calculation. 
\end{proof}
We turn to the other cases.
\begin{lemma}\label{4.5}
For $ s,t\geq 1$ the product of classes $\delta[a]\in H^s \Omega^*,\delta[b]\in H^t\Omega^*$ vanishes if and only if there is an $r\in A_\Q^{\otimes {s+t-1}}$ with 
$$(a \, db +dr) \in \Gamma^{s+t-1}.$$
We write $ r=r(a, b)$ for any such element.
\end{lemma}
\begin{proof}
If $\delta[a]\delta[b]=0$  there is $x\in \Gamma^{s+t-1} $ with
$$ dx =da \, db.$$
Applying $\tau 1\otimes \ldots \otimes 1$ gives
$$ x-dx'= (a-da')db$$
and we set $r=x'-a'\ db$. Conversely, we have
$$d(a\, db +dr)= da\, db.$$
\end{proof}
\begin{prop}\label{4.6}
Suppose that the Massey product  $\left< \delta [a],\delta [b],\delta [c]\right>$ is defined. Then 
\begin{enumerate}
\item
for $| a|=|b|=|c|=0$ it holds
$$ \delta [a(b\,dc + dr(b,c))+r(a,b)dc]\in \left< \delta [a],\delta [b],\delta [c]\right>$$
\item
for $| a|=1$, $|b|=0$ and $c=q\in\Z$ it holds
$$ \delta [q(ab+r(a,b))]\in \left< \delta [a],\delta [b],q\right>.$$
\item
for $ a=q\in \Z$, $|b|=0$ and $|c|=1$ it holds
$$ \delta [-qr(b,c))]\in \left< q,\delta [b],\delta [c]\right>.$$
\item
If $\left< q,\delta [b],\delta [c]\right>$ with $|b|=1$, $|c|=0$ is defined and $b$ is $q$-adapted it holds
$$ \delta [qr(b,c)]\in \left< q,\delta [b],\delta [c]\right>$$
\item
if $\left< \delta [a],\delta [b],q\right>$ with $|b|=1$, $|a|=0$ is defined and $b$ is $q$-adapted it holds
$$ \delta [-q(ab+r(a,b))]\in \left< \delta [a],\delta [b],q\right>.$$
\end{enumerate}
\end{prop}
\begin{proof}
By the lemma a representative of the first Massey product is 
$$ da\, (b\, dc +dr(b,c))+(a \, db +dr(a,b))\, dc.$$
Applying $\tau 1\otimes 1$ gives
$$a(b\, dc+dr(b,c))+r(a,b)\, dc$$
as claimed. For $(ii)$ a representative is
$$da\, qb+(a\, db +dr(a,b))q.$$ In the last case it is 
$$-q(b\, dc+dr(b,c)) +qb\, dc.$$
Again applying $\tau 1 \otimes 1$ gives the results up to boundary terms.
The other Massey products are obtained with the help of the lemma as before. In the case $(iv)$ one gets
$$ q(b \, dc +dr) +(-qb) dc=q\, dr $$ and in the last case
$$ -da \, qb+(-a\, db-dr)q= -q(d(ab)+dr).$$
\end{proof}
We close this section with an example of a 4-fold Massey product. Note that in principle other types and even higher Massey products can be calculated with the same method.  
\begin{prop}\label{4.7}
Suppose $a$ and $b$ have degree 1 and are $q$-adapted. Then the following conditions are equivalent:
\begin{enumerate}
\item
the Massey product $\left< q,\delta [a],q,\delta [b]\right>$ 
is defined 
\item
$0\in \left< \delta [a],q,\delta [b]\right>$
\item 
there are $r\in A_\Q\otimes A_\Q$ and $dv,dw\in \Gamma$ with  
$ qab+a\, dw + dv\, b+dr\in \Gamma^{\otimes 2}.$
\end{enumerate}
In this case the 4-fold product
contains $\delta (-qr+v\, dw)$. Moreover, the same assertions hold for $\left< \delta [a],q,\delta [b],q\right>$.
\end{prop}
\begin{proof}
A defining system can look like
$$ \xymatrix@!=0,1pc{q&&da&&q&&db\\&-qa-dv&&-qa-dv&&-qb-dw&\\
&&0&&x&&}$$
In particular a defining system exists if there is a $x\in \Gamma^{\otimes 2 }$ with $$dx = qd(ab)+dv\, db+da\, dw.$$ Applying $\tau 1\otimes1\otimes 1$ gives $x=qab+dv\, b +a\, dw+dr $ for some $r$. This shows the equivalences. Finally, if $x$ is as above the missing corner is $-qr+v\, dw$ up to a boundary. 
\end{proof}

\section{Invariants in modular forms}\label{invariants}
For a $\Z [1/N ]$-algebra $R$  let $M_k (\Gamma_1 (N ))_R$ be the ring of $\Gamma_1(N)$ modular 
forms of weight $k$ over $R[\zeta_N]$ which are meromorphic at the cusps. 
Let $TMF_1 (N )$ denote the corresponding spectrum of topological modular forms. Its coefficients
$\pi_*   TMF_1(N )$ are concentrated in even degrees, and we have 
$$\pi_{2k}TMF_1 (N )  \cong M_k (\Gamma_1 (N )).$$ The spectrum $TMF_1(N)$ is complex orientable with formal group isomorphic to the formal completion of the universal elliptic curve over the ring of $\Gamma_1(N)$ modular forms. In fact, it is Landweber exact and also carries the names $Ell^{\Gamma_1(N)}$ or $E^{\Gamma_1(N)}$ in the older literature (see \cite{MR1071369} \cite{MR1271552} for theories away from the prime 2 and  \cite{MR1235295} \cite{MR1660325} for the general case). Nowadays one uses the notation $TMF_1(N)$ since it is a complex oriented relative of the spectrum $TMF$ of topological modular forms. It can be obtained as global sections of a sheaf of spectra over some moduli space of elliptic curves and level structures (see 
\cite{MR2469520} et al.).
Since the congruence subgroup $\Gamma_1(N)$ will be fixed once and for all we will remove it from the notation and write $TMF$ instead of $TMF_1(N)$.\par
Let $D=D(\Gamma_1(N))$ be the ring of divided congruences (cf.\cite{MR0417059}). An element in 
$D$ is a sum $\sum f_k$ of modular forms $f_k\in (M_k)_\Q$ with an integral $q$-expansion, that is, an expansion with coefficients in $\Z[1/N,\zeta_N]$.
The following result is well known (see \cite{MR1307488}  for $p\not=2,3$ or \cite{MR1781277} \cite{MR2076927} for the general case.) 
\begin{thm}
The $K$-homology of $TMF$ vanishes in odd degrees and  there is an isomorphism
$$ \pi_0 K \wedge TMF\cong D.$$
\end{thm}
Set $E=TMF$. Consider the resolution of $E_*$ as $E_*E$ comodule
$$\xymatrix@!=2.5pc{& E_*\otimes \Q \ar@{->>}[rd]\ar@{-->}[rr]&&E_*\Sigma K \otimes\Q/Z \ar@{-->}[rr]\ar@{->>}[rd]
&&E_*\Sigma \bar{K} \otimes \Q/\Z\\
E_*\ar@{>->}[ru]
&& E_*\otimes \Q /\Z \ar@{>->}[ru]&&E_*\Sigma \bar{K} \otimes \Q/\Z \ar[ru]^= 
}$$
The middle sequence is exact 
since $E_*=M_*$ is a pure subgroup of $E_*K=D$. The short exact sequences provide connecting homomorphisms $\delta$ for the extension groups. 
\begin{prop}\label{image}
Let $P(M)=Hom_{E_*E}(E_*,M)$ denote the primitives of a comodule $M$.
Then the sequence
$$ \xymatrix{ \pi_*\Sigma \bar{K}\otimes \Q/\Z \ar[r]& P(E_*\Sigma \bar{K}\otimes \Q / \Z)\ar[r]^{\delta^2} & \mbox{Ext}^2_{E_*E}(E_*,E_*)\ar[r]&0}$$
is exact. Moreover the $f$-invariant is the composite of $e_2$ with the inclusion 
$$  \xymatrix{ \mbox{Ext}^{2,2n}_{E_*E}(E_*,E_*)\cong  P(E_{2n-1}\bar{K})/  \pi_{2n-1}\bar{K}\otimes \Q/\Z \ar@{>->}[r] &(D/M_0+M_{n})_{\Q  / \Z}} $$ 
\end{prop}
\begin{proof}
It is not hard to see that the cohomology of $E_*\otimes \Q$ is concentrated in degree and dimension 0. In fact, the map $\tau$ gives a contracting homotopy. Hence, for positive dimension we have the isomorphism
$$ \delta: \xymatrix{ \mbox{Ext}^1(E_*, E_*\otimes \Q / \Z) \ar[r] & \mbox{Ext}^2(E_*,E_*)}.$$
Furthermore, the middle short exact sequence of the resolution gives the exact sequence
$$ \xymatrix{P(E_*K\otimes \Q / \Z) \ar[r] & P(E_*\Sigma \bar{K}\otimes \Q / \Z)\ar[r]^{\delta} & \mbox{Ext}^1(E_*,E_*\otimes \Q/\Z )}.$$
We claim that $\delta $ is surjective in positive dimensions. For that it suffices to show that the map into $\mbox{Ext}^1(E_*,E_*K\otimes \Q / \Z)$ vanishes. This map admits a factorization
$$  \xymatrix{\mbox{Ext}^1(E_*,E_*\otimes \Q/\Z )\ar[r]& \mbox{Ext}^1(E_*,E_*E\otimes \Q/\Z )\ar[r]&\mbox{Ext}^1(E_*,E_*K\otimes \Q / \Z)}$$
induced by the ring map $\chi_0: E\ra K$ in $E$-homology. Since the middle term vanishes we have shown that $\delta$ is surjective. \par
It remains to identify the first map. We  claim that the map
$$\xymatrix{ \pi_*K\otimes \Q / \Z \ar[r] & P(E_*K\otimes \Q / \Z)}$$
is an isomorphism. Let $K_T$ be the elliptic theory associated to the Tate curve. Its coefficients are integral Laurent series in a variable $q$ in all even degrees and they vanish in odd degrees.  There is the Miller character, that is, a ring map 
$\chi: E \ra K_T$ which is the $q$-expansion map on coefficients. Consider the injective map
$$\xymatrix {(\chi \wedge 1 )_* :E_*K \otimes \Q/ \Z \ar[r]& {K_T}_*K \otimes \Q/ \Z}$$ which takes $q$-expansions. A primitive in its source gives a primitive in the target if the target is viewed as a comodule over ${K_T}_*K_T$. The primitives in the target all lie in the primitive of the extended comodule ${K_T}_*K_T\otimes \Q/Z$ and hence in $\pi_*K_T\otimes \Q/\Z$. Since they also lie in ${K_T}_*K\otimes \Q /\Z$ they must lie in $(\pi_*K)\otimes{\Q / \Z}$. This group coincides with $\pi_*(\Sigma \bar{K})\otimes {\Q/ \Z}$ in positive dimension  and hence  the claim follows.  
\par
The $f$-invariant is the composite of $e_2$ with
$$ \xymatrix{H^2\Omega_*  \ar@{>->}[r] & \Omega_2/d\Omega_1\ar[r]^{\! \!\! \!\! \! \! \!\! \!\tau 1\otimes 1 }\ar[r]&{\Gamma}/d(M_n)\otimes \Q / \Z  \ar[r]^{1\wedge \chi_0 } & D/(M_0+M_n)_{\Q / \Z} }.$$
As we have seen before the second map gives an inverse of the connecting homomorphism. Hence the claim follows from   the commutative diagram
$$ \xymatrix{ H^1(\Omega_{\Q / \Z})&P(E_*\Sigma \bar{K}\otimes \Q /\Z)/ \sim  \ar[r] \ar[l]_\delta ^\cong &  D/(M_0+M_n)_{ \Q / \Z} \\H^1(\Omega_{ \Q / \Z}) \ar[u]_=&  P(\bar{\Gamma} \otimes \Q /\Z)/\sim  \ar[r]\ar[u]_{1\wedge \chi_0} \ar[l]_\delta^\cong  & {\Gamma}/d(M_n)_{\Q / \Z} \ar[u]_{1\wedge \chi_0}} $$
in which a 1-cycle in the cobar complex of the left lower corner  is send to itself under the composite of the bottom row. 
\end{proof}
\begin{remark}
We will not study the primitives of
$$E_*\Sigma \bar{K}= (D/M_n)_{ \Q / \Z}$$ here since it is not needed for Toda brackets and the rest of this work. However, we mention that the primitives are eigenforms under the action of the Hecke operations and they are fixed under the action of the Adams operations (see \cite[1.1]{MR1692001}). This suggests that the elements of the 2-line in the cokernel of $J$ are related to newforms. A precise relationship locally at primes $p\geq 5$ between the 2-line and certain $p$-adic modular forms is described in  \cite{MR2469520}.
\end{remark}
The following result is due to Bunke and Naumann. It is very useful for explicit calculations. 
\begin{lemma}\label{Naumann}
\begin{enumerate}
\item
There exists a $\Z$-basis $f_0\ldots f_{n_k}$ of $M_{k}$ 
such that  $q^i (f_j ) = \delta_{ij}$ for all $i,j$.
\item
For a basis as above, the map
$$  \xymatrix{(D/M_k)_{\Q / \Z } \ar[r] & \prod_{i\geq n_k}\Q/\Z}, \qquad f\mapsto 
\left( q^\nu(f-\sum_{i=0}^{n_k-1}q^i(f)f_i)\right)_{\nu\geq n_k}$$
is injective.
\end{enumerate}\end{lemma}
\begin{proof}
The first part is lemma 9.2 of \cite{BN10} and the second part is stated there in terms of $q$-expansions.
\end{proof}
Clearly, it suffices to check finitely many Fourier coefficients for a modular form in the source. There are upper bounds for this number but we do not work them out here.
\section{The $f$-invariant of Toda brackets}
In this section we apply our formulas to the $f$-invariant of Toda brackets. We start with the simplest case.
\begin{thm} 
\label{thm1}
Suppose that the Toda bracket $\left< \alpha , p \iota , \beta \right>  $  is defined for some $\alpha, \beta $ in positive dimensions $m,n$ and some $p\in \Z$. 
\begin{enumerate}
\item
Let the $f$-invariant of $\alpha$ be defined. Choose a representative of $f(\alpha)$ whose $q$-expansion is annihilated by $p$. Then we have
$$ f\left< \alpha , p \iota , \beta\right> =p e(\beta )f(\alpha ) $$
modulo $e(\beta)H^{0,m+2}(\Omega^*_{\Q /\Z})$.
\item
Let the $f$-invariant of $\beta$ be defined. Choose a representative of $f(\alpha)$ whose $q$-expansion is annihilated by $p$. Then we have
$$  f\left< \alpha , p \iota , \beta\right> = -p e(\alpha )f(\beta ) $$
modulo $e(\alpha)H^{0,n+2}(\Omega^*_{\Q /\Z})$.
\end{enumerate}
\end{thm}
\begin{proof} First note in view of \ref{indet} that the indeterminacy in the theorem coincides with the image under the $f$-invariant of the Massey product indeterminacy.
For $(i)$ we can choose a $p$-adapted representative $a$ with $\delta [a]=e_2(\alpha)$. Its image under $\chi^0_*$ coincides with the normalized representative  of $f(\alpha)$ up to a constant and a modular form $g$ of weight $m/2 +1$  for which $pg$ is a $\Q/\Z$-cycle. Hence it follows from \ref{4.4} $(i)$with $\delta[c]=e_1\beta$
$$  f\left< \alpha , p \iota , \beta\right> \ni \chi^0_* [pac]= p f(\alpha)e(\beta)+pg e(\beta)$$
which is the first claim.
This shows $(i)$ and $(ii)$ is analogues. 
\end{proof}
Other Toda brackets are more complicated. We first need the
\begin{definition}\rm  
A divided congruence $f$ in $k$ variables  has {\em virtual weight $n$} if there is a modular form $g$ of weight $n$  with the same $\Q /\Z $ Fourier coefficients $a_{i_1,i_2\ldots i_k}$ for all $i_1,i_2,\ldots , i_k >0$. We write $[f]_n$ for such a $\Q / \Z$-modular form $g$. 
\end{definition}

\begin{lemma}\label{6.3}
Suppose $f$ has virtual weight $n$.
\begin{enumerate}
\item
For $k=1$ any two modular forms in the bracket $[f]_n$ only differ by cycles in $\Omega^{0,2n}_{\Q /\Z}$. \item
For $k=2$ any two modular forms in the bracket only differ by  cycles in $\Omega_{\Q /\Z}^1$ and modular forms of the form $g \otimes 1 $ with $g$ of weight $n$.
\end{enumerate}
\end{lemma}
\begin{proof}
For $(i)$ two elements in $[f]$ differ by modular forms $h$ of weight $n$ with integral expansion except the 0-coefficient. 
Any such $h$ is a cycle in $\Omega^{0,2n}_{\Q / \Z}$. We remark for $n$ even and level 1 that the highest denominator in the constant coefficient appearing among all such $h$ is the divided Bernoulli number which happens if $h$ is the divided Eisenstein series $\bar{E}_n$ (see Serre \cite{MR0404145}). For the case of two variables one observes that for any difference modular form $h=\sum h_1\otimes h_2$ it holds
$$\sum h_1\otimes h_2 -h_1h_2\otimes 1 +h_2\otimes h_1 -1\otimes h_1h_2=0$$
(cf.\ \cite[Eq.3.2]{MR1660325}). Hence, $z=h-\sum h_1h_2\otimes 1$ is antisymmetric with vanishing  $q_L^iq_R^j$-coefficients for $i,j>0$. This implies that $z$ is a cycle as one easily verifies.
\end{proof}
In the following we write $e_{M}(\alpha)$ for any representative of $\delta^{-1}e_1(\alpha)$ in $\Omega^0_{\Q / \Z} $ and
$$ e(\alpha)=q^0(e_M(\alpha)).$$
 Note that for  $|\alpha|=4n-1$ and level 1 modular forms with 6 inverted we have
$$e_M(\alpha)=e(\alpha)E_n$$
and $e$ is the classical $e$-invariant. 

\begin{thm}
\label{thm 2}
 Suppose that the Toda bracket $\left< \alpha , \beta , \gamma \right>  $  is defined.
\begin{enumerate}
\item
Let $|\alpha|=2k-1$, $|\beta |=2l-1 $ and $|\gamma |=2m-1$. Then the modular forms $f(\alpha,\beta)=[e_M(\alpha)e(\beta)]_{k+l}$ and $f(\beta,\gamma)=[e_M(\beta)e(\gamma)]_{l+m}$  exist and we have
$$  f\left< \alpha , \beta, \gamma \right> =  
e_M(\alpha) (q^0( f(\beta,\gamma))-e(\beta)e(\gamma)) +e(\gamma)f(\alpha, \beta)
$$
modulo the indeterminacy $ e_M(\alpha) q^0H^{0,2(l+m)}(\Omega^*_{\Q / \Z})+  e(\gamma)H^{0,2(k+l)}(\Omega^*_{\Q / \Z})$. 
\item
Let $|\alpha|=2k-2, |\beta|=2l-1$. Then $[ f(\alpha)\otimes e_M(\beta)]_{k+l}$ exists and we have
$$  f\left< \alpha ,  \beta, p \iota \right>  = p f(\alpha) e(\beta) + p\chi^0_*  [f(\alpha)\otimes e_M(\beta)]_{k+l}$$
modulo $pP(({D/M_{k+l}})_ {\Q / \Z})$. 
\item
Let $|\beta|=2l-1,|\gamma|=2m-2$. Then $[e_M(\beta)\otimes f(\gamma)]_{l+m}$ exists and we have
$$  f\left<  p \iota ,  \beta, \gamma \right>  = -p\chi^0_*[e_M(\beta)\otimes f(\gamma)]_{l+m}$$
modulo $pP(({D/M_{l+m}})_{\Q / \Z})$.
\item
Let $|\alpha|=2k-1$ and $|\beta|=2l-2$.  Choose a representative of $f(\beta)$ whose $q$-expansion is annihilated by $p$. Then the modular form $[e_M(\alpha )\otimes f(\beta)]_{k+l}$ exists and we have
$$  f\left<\alpha  ,  \beta,   p \iota\right>  = -p e_M(\alpha) f(\beta) -p \chi^0_*  [e_M(\alpha)\otimes f(\beta)]_{k+l}$$
modulo $pP(({D/M_{k+l}})_{\Q / \Z})$. 
\item
Let $|\beta|=2l-2$ and $|\gamma|=2m-1$.  Choose a representative of $f(\beta)$  whose $q$-expansion is annihilated by $p$. Then the modular form $[f(\beta)\otimes e_M(\gamma)]_{l+m}$ exists and we have
$$  f\left<  p \iota ,  \beta, \gamma \right>  = p\chi^0_* [f(\beta)\otimes e_M(\gamma)]_{l+m}$$
modulo $pP(({D/M_{l+m}})_{\Q / \Z})$. 
\end{enumerate}
\end{thm}
\begin{proof}
For $(i)$  $\alpha$, $\beta$, $\gamma$ is represented by $a=e_M(\alpha)$, $b=e_M(\beta)$, $c=e_M(\gamma)$ in the $\Q / \Z$-cobar complex. Since the product $da\, db$ vanishes there is a modular form $r$ of degree $k+l$ with $a\, db+dr\in \Gamma=E_*E$. Applying $\chi^0_*$ gives
$$ a q^0(b)-ab + q^0(r)-r \in D.$$
Thus $aq^0(b) $ is congruent to the modular form $ab+r$ of degree $k+l$ up to a constant and
$$ f(\alpha,\beta)=[e_M(\alpha ) e(\beta)]_{k+l}=ab+r$$ modulo cycles of degree $k+l$.
The analogues statement for $b,c$ shows that the brackets exist. 
The $f$-invariant of the Toda bracket is hence obtained from \ref{4.6}$(i)$ and the computation
\begin{eqnarray*}
f\left< \alpha , \beta, \gamma \right> &=&  \chi_*^0\delta^{-1} \left< \delta [a],\delta [b],\delta [c]\right>\\
&=&   \chi_*^0  [a(b\,dc + dr(b,c))+r(a,b)dc]\\
&=& ab(q^0(c)-c)+a(q^0(r(b,c)-r(b,c))+r(a,b)(q^0(c)-c)\\
&=& abq^0(c)+aq^0(r(b,c))+r(a,b)q^0(c)\\
&=&aq^0(f(\beta, \gamma)-bc)+f(\alpha ,\beta )q^0(c)\\
&=& e_M(\alpha)(q^0( f(\beta,\gamma))-e(\beta)e(\gamma)) +e(\gamma)f(\alpha, \beta)
\end{eqnarray*}
This is the result.

For $(ii)$ we find with lemma \ref{4.5} a modular form $r$ of weight $k+l$ with 
$a\, db +dr\in \Gamma^2$. This is an expression with Fourier coefficients in variables $q_L$,$q_M$ and $q_R$. For $i,k>0$ we have
\begin{eqnarray*}
q_L^iq_M^0q_R^k(a\, db+dr)&=& \sum q^i(a_1)q^0(a_2)q^k(b)-q^i(r_1) q^k(r_2)\\&=&q_L^iq_R^k(\sum a_1q^0(a_2)\otimes b-r_1\otimes r_2)
\end{eqnarray*}
which hence is integral. This shows 
$$ r=[f(\alpha)\otimes e_M(\beta)]_{k+l}.$$
We compute with \ref{4.6}$(ii)$
\begin{eqnarray*}
 f\left< \alpha ,  \beta, p \iota \right>  &=& \chi^0_* \delta^{-1} \left< \delta [a],\delta [b],p\iota \right>\\
 &=& \chi^0_* (p(ab+r))\\
 &=& p f(\alpha) e(\beta) + p \chi^0_*  [f(\alpha)\otimes e_M(\beta)]_{k+l}
 \end{eqnarray*}
 Similarly, for $(iii)$ we have 
 \begin{eqnarray*}
 f\left< p \iota  ,  \beta, \gamma \right>  &=& \chi^0_* \delta^{-1} \left< p \iota ,\delta [b],\delta \gamma \right>\\
 &=& \chi^0_* (-p r))\\
 &=& - p \chi^0_*  [ e_M(\beta) \otimes f(\gamma)]_{l+m}
 \end{eqnarray*}
 and for $(iv)$
 \begin{eqnarray*}
  f\left< \alpha ,  \beta, p \iota \right>  &=& \chi^0_* \delta^{-1} \left< \delta [a],\delta [b],p\iota \right>\\
 &=& -\chi^0_* (p(ab+r))\\
 &=& - p e_M(\alpha)  f(\beta) -p \chi^0_*  [e_M(\alpha) \otimes f(\beta)]_{k+l}
\end{eqnarray*}
Finally,
 \begin{eqnarray*}
 f\left< p \iota  ,  \beta, \gamma \right>  &=& \chi^0_* \delta^{-1} \left< p \iota ,\delta [b],\delta \gamma \right>\\
 &=& \chi^0_* (p r))\\
 &=& p \chi^0_*  [ f(\beta) \otimes e_M(\gamma)]_{l+m}
 \end{eqnarray*}
The indeterminacy is readily verified with \ref{indet}.
\end{proof}
\begin{thm}\label{4toda}
Suppose that the Toda bracket $\left< p \iota, \alpha, p\iota , \beta \right> $ is strictly defined and let $|\alpha|=2k-2, |\beta |=2l-2$. Then there are  representatives of  $f(\alpha)$ and $f(\beta)$ which are annihilated by $p$ and for which the bracket $[pf(\alpha)\otimes f(\beta)]_{k+l}$ exists. For any such modular form we have
$$ f\left< p \iota, \alpha, p \iota , \beta \right> = p\chi^0_*[pf(\alpha)\otimes f(\beta)]_{k+l}$$
modulo $$\chi^0_*([H^{0,2k}(\Omega^*_{\Q/\Z})\otimes f(\beta )]_{k+l}+[f(\alpha)\otimes H^{0,2l}(\Omega^*_{\Q / \Z})]_{k+l})+pq^0(H^{0,2k}(\Omega^*_{\Q / \Z}))f(\beta) $$ and the indeterminacies coming from the 3-fold brackets
$$  p(P((D/M_{k+l})_{\Q / \Z}) + H^{0,2k}(\Omega^*_{\Q / \Z})q^0(H^{0,2l}(\Omega^*_{\Q / \Z})). $$
The same holds for the bracket $\left<\alpha, p\iota , \beta , p \iota \right> $.
\end{thm}
\begin{proof}
With \ref{4.3} we find $p$-adapted $a,b$ with $e_2(\alpha)=\delta[a]$ and $e_2(\beta )=\delta[b]$. Without loss of generality we can assume $pab+dr\in \Gamma^{\otimes 2 }$ (else replace $a$ by $a+dv/p$ and $b$ by $b+dw/p$). 
In particular for $i,j>0$ the numbers 
\begin{eqnarray*}
q_L^iq_M^0q_R^j(p\sum a_1\otimes a_2b_1\otimes b_2+dr)&=& p\sum q^i(a_1)q^0(a_2b_1)q^j(b_2)-q^i(r_1)q^j(r_2)\\&=& q^i_Lq_R^j(p\sum a_1q^0(a_2)\otimes q^0(b_1)b_2-r) 
\end{eqnarray*}
are integral. Since $f(\alpha)=q^0(a_2)a_1$ and $f(\beta)=-q^0(b_1)b_2$ we conclude that the bracket $[pf(\alpha)\otimes f(\beta)]_{k+l}$ exists. Its indeterminacy is as in \ref{6.3}$(ii)$.  Hence we have with \ref{4.7}
 \begin{eqnarray*}
 f\left< p \iota, \alpha, p \iota , \beta \right>  &=& \chi^0_* \delta^{-1} \left< p \iota, \delta[a],  p \iota ,\delta [b] \right>\\
 &=& - \chi^0_* p r\\
 &=& p \chi^0_* [pf(\alpha)\otimes f(\beta)]_{k+l}
 \end{eqnarray*}
The indeterminacy uses \ref{indet} and the calculations of the 3-fold brackets above.
\end{proof}
\section{Examples}
In dimension 8 there is the Toda bracket $\left< \nu^2 , 2 ,\eta \right>$ where $\nu$ is the Hopf map of dimension 3. We use the formula \ref{thm1} for level 3 TMF to show that this class coincides with $\beta_2$: The $f$-invariant of the product $\nu^2$ can be computed from the formula (cf.\cite{vB09a})
$$ f(\nu^2) = e(\nu)e_M(\nu)= \frac{E_1^2}{12^2}.$$
which can be normalized to $-1/2 ((E_1^2-1)/12)^2$. Hence we have 
$$ f\left< \nu^2 , 2 ,\eta \right>=-\frac{1}{2} \left( \frac{E_1^2-1}{12}\right)^2$$
which coincides with the $f$-invariant of $\beta_2$ (see \cite{vB09c}).\par
Similarly, we have for the dimension 7 Hopf map $\sigma$ instead of $\nu$
$$ f(\sigma^2) = e(\sigma)e_M(\sigma)= \frac{E_4}{240^2}= -\frac{1}{2}\left( \frac{E_4-1}{240}\right)^2$$
and hence
$$ f\left< \sigma^2 , 2 ,\eta \right>=-\frac{1}{2} \left( \frac{E_4-1}{240}\right)^2.$$
The computation in \cite{vB09a}p.7 shows that this expression is congruent to 
$$f(\beta_{4/3})=\frac{1}{2} \left( \frac{E_1^2-1}{4}\right)^4+\frac{1}{2} \left( \frac{E_1^2-1}{4}\right)^3.$$ There is another way to compute this class using the Toda relation
$$\left< \sigma^2 , 2 ,\eta \right>=\left< \sigma , 2\sigma  ,\eta \right>$$
and formula \ref{thm 2}$(i)$: First we have
$$ f(\sigma , 2 \sigma) = \left[ 2 \frac{E_4}{240^2}\right]_8=  \frac{E_4^2}{240^2}$$
and
$$ f(2 \sigma ,  \eta ) = \left[ \frac{E_4}{240}\right]_5=  0.$$
This gives
$$
 f\left< \sigma , 2 \sigma,\eta \right>= -\frac{E_4}{240^2}+\frac{1}{2}\frac{E_4^2 }{240^2}=\frac{1}{2} \left( \frac{E_4-1}{240}\right)^2=f(\beta_{4/3})
$$

In dimension 18 there is the class $\left< \sigma , 2 \sigma, \nu \right>$ for which the formula reads
$$f\left< \sigma , 2 \sigma, \nu \right>=   \frac{E_4}{240}\left(q^0 \left[ \frac{E_4}{1440}\right]_6-\frac{1}{1440}\right)+\frac{1}{12} \frac{E_4^2}{240^2}
$$
To evaluate the bracket observe that
$$d^3=d^5 \mbox { mod } 8$$
for all integers $d$ and hence
$$\frac{E_4-1}{240}=\sum_{n\geq 1}\sum_{d\mid n}d^3q^n=\sum_{n\geq 1 }\sum_{d\mid n}d^5q^n =\frac{1-E_6}{504}\mbox{ mod } 8.$$
This gives
\begin{eqnarray*}
f\left< \sigma , 2 \sigma, \nu \right>&=&  \frac{E_4}{240}\left(-\frac{1}{3024}-\frac{1}{1440}\right)+\frac{1}{12} \frac{E_4^2}{240^2}\\
&=&\frac{1}{6}\frac{1}{240^2}\left( -\frac{31}{21}E_4+\frac{1}{2}E_4^2\right)
\end{eqnarray*}
which has order 4 modulo indeterminacy. In fact with lemma \ref{Naumann} one can show that it coincides with $f(-\beta_{4/2,2})$. 
\par
The Toda bracket $\left< \sigma^2, 2 ,\sigma^2 , 2\right>$ in dimension 30 exists (see \cite{MR810962}1.2 and 1.3.)
Compute modulo indeterminacy
\begin{eqnarray*}
\lefteqn{\left[ \frac{1}{2}\left( \frac{E_4 -1}{240}\right)^2 \otimes \left( \frac{E_4 -1}{240}\right)^2 \right]_{16} }\\
&=& \left[ \frac{1}{240^4}\left( \frac{E_4^2\otimes E_4^2}{2}-E_4^2\otimes E_4-E_4\otimes E_4^2 +2 E_4 \otimes E_4\right) \right]_{16}\\
&=& \frac{1}{240^4}\left(\frac{E_4^2\otimes E_4^2}{2}-\frac{E_4^3\otimes E_4}{3}-\frac{E_4\otimes E_4^3}{3}\right) \\
&=& \frac{1}{12}\left(\frac{E_4 \otimes 1 -1 \otimes E_4}{240}\right) ^4
\end{eqnarray*}
Here we used in the second step that
$$ \frac{1}{3} \left( \frac{E_4-1}{240}\right) ^3 \otimes \frac{E_4-1}{240}$$
is integral. This gives with \ref{4toda}
\begin{eqnarray*}
f\left< \sigma^2, 2 ,\sigma^2 , 2\right>&=& 2 \chi^0_*\left[ \frac{1}{2}\left(\frac{E_4 -1}{240}\right)^2\otimes \left(\frac{E_4 -1}{240}\right)^2\right]_{16}\\
&=& 2 \chi^0_*\left(   \frac{1}{12}\left(\frac{E_4 \otimes 1 -1 \otimes E_4}{240}\right) ^4 \right)\\
&=& \frac{1}{2} \left(\frac{E_4 -1}{240}\right)^4
\end{eqnarray*}
which has order 2 and coincides with the $f$ invariant of  the Kervaire class $f(\beta_{8/8})$ (compare \cite{vB09c}).\bibliographystyle{amsalpha}

\bibliography{toda}
\end{document}